\documentclass[11pt,twoside]{amsart}
\usepackage{mathrsfs}
\usepackage{amsmath}
\usepackage{amsthm}
\usepackage{amsfonts}
\usepackage{amssymb}
\usepackage{latexsym}
\usepackage[all]{xy}
\usepackage{graphicx}

\date{\empty}
\thanks{}
\allowdisplaybreaks[4] \footskip=15pt
\renewcommand{\uppercasenonmath}[1]{}

\numberwithin{equation}{section} \theoremstyle{plain}
\newtheorem*{thm*}{Main Theorem}
\newtheorem{theorem}{Theorem}[section]
\newtheorem{corollary}[theorem]{Corollary}
\newtheorem*{corollary*}{Corollary}

\newtheorem*{claim*}{Claim}
\newtheorem{lemma}[theorem]{Lemma}
\newtheorem*{lemma*}{Lemma}
\newtheorem{proposition}[theorem]{Proposition}
\newtheorem*{proposition*}{Proposition}
\newtheorem{remark}[theorem]{Remark}
\newtheorem*{remark*}{Remark}
\newtheorem{example}[theorem]{Example}
\newtheorem*{example*}{Example}

\newtheorem*{question*}{Question}
\newtheorem{definition}[theorem]{Definition}
\newtheorem*{definition*}{Definition}

\newtheorem*{acknowledgements*}{ACKNOWLEDGEMENTS}

\def\sbmatrix{\left[\begin{array}}
\def\endsbmatrix{\end{array}\right]}

\begin{document}
\begin{center}
{\large  \bf New results on EP elements in rings with involution}\\
\vspace{0.8cm}   Long Wang$^{1}$\footnote{Corresponding author. E-mail: lwangmath@yzu.edu.cn.
The research is supported by the Ministry of Education, Science and Technological Development,
Republic of Serbia, grant no. 174007,
NSF of China (11901510),
NSF of Jiangsu Province of China (BK20170589),
China Postdoctoral Science Foundation Funded Project (2017M611920).}, Dijana Mosi\'{c}$^{2}$ and Yuefeng Gao$^{3}$\\
\vspace{0.5cm} {\small $^{1}$School of Mathematical Sciences, Yangzhou University,
Yangzhou 225002, China\\
$^{2}$Faculty of Sciences and Mathematics, University of Ni\v{s}, Ni\v{s}, Serbia\\
$^{3}$College of Science,  University of Shanghai for Science and Technology,
Shanghai 200093, China}
\end{center}

\bigskip

{ \bf  Abstract:}  \leftskip0truemm\rightskip0truemm
In this paper, we mainly give characterizations of EP elements in terms of equations.
In addition,  a related  notion named a central EP element is defined and investigated.
Finally, we focus on characterizations of a generalized EP element, i.e. *-DMP element.
 \bigskip


\section {  Introduction}

Let $\mathcal{R}$ be a ring with unit 1.
Recall that the Drazin inverse \cite{MPD1958} of $a \in \mathcal{R}$ is the element $x \in \mathcal{R}$ which satisfies
\begin{center}
$xax = x$,\quad $ax = xa$,\quad $a^{k} = a^{k+1}x$\quad for some $k \geq 0$.
\end{center}
The element $x$ above is unique if it exists and is denoted by $a^{D}$.
The least such $k$ is called the index of $a$, and denoted by $\text{ind}(a)$.
In particular, when $\text{ind}(a)=1$,
the Drazin inverse $a^{D}$ is called the group inverse of $a$ and it is denoted by $a^{\sharp}$.
The set of all Drazin (resp. group) invertible elements of $\mathcal{R}$ is denoted by $\mathcal{R}^{D}$ (resp. $\mathcal{R}^{\sharp}$).
In a ring $\mathcal{R}$, an involution $\ast: \mathcal{R} \rightarrow \mathcal{R}$
is an anti-isomorphism which satisfies $(a^{\ast})^{\ast} = a$, $(a + b)^{\ast} = a^{\ast} + b^{\ast}$ and
$(ab)^{\ast}= b^{\ast}a^{\ast}$ for all $a, b \in \mathcal{R}$. $\mathcal{R}$ is called a $\ast$-ring if $\mathcal{R}$  is a ring with an involution $\ast$.
In what follows, $\mathcal{R}$ is a $\ast$-ring.
An element $a \in \mathcal{R}$ is said to be Moore-Penrose invertible if the following equations:
\begin{center}
(1) $axa = a$, (2) $xax = x$, (3) $(ax)^{\ast} = ax$ and (4) $(xa)^{\ast} = xa$
\end{center}
have a common solution. Such a solution is unique if it exists,
and is denoted by $a^{\dag}$ as usual.
The set of all Moore-Penrose invertible elements of $\mathcal{R}$ will be denoted by $\mathcal{R}^{\dag}$.
If $x \in \mathcal{R}$ satisfies both Eqs. (1) and (3),
then $x$ is called  a $\{1, 3\}$-inverse of $a$ and is denoted by $a^{(1,3)}$.
The set of all $\{1, 3\}$-invertible elements of $\mathcal{R}$ is denoted by $\mathcal{R}^{\{1,3\}}$.
Similarly, if $x \in \mathcal{R}$ satisfies both Eqs. (1) and (4),
then $x$ is called a $\{1, 4\}$-inverse of $a$ and is denoted by $a^{(1,4)}$.
The set of all $\{1, 4\}$-invertible elements of $\mathcal{R}$ is denoted by $\mathcal{R}^{\{1,4\}}$.

The notion of core inverse was introduced by Baksalary and Trenkler \cite{BT2010} for a complex matrix in 2010.
Then, the notion of  core inverse was generalized to $\ast$-rings by Raki\'{c} et al.\cite{RDD2014}.
Xu, Chen and Zhang \cite{XCZ1} characterized core invertible elements in $\ast$-rings by three equations.
Let $a\in \mathcal{R}$. If there exists $x\in \mathcal{R}$ such that the following three equations hold:
\begin{equation}
 xa^2=a,~ax^2=x~\text{and}~(ax)^{\ast}=ax,
\end{equation}
then $a$ is called to be core invertible, $x$ is a core inverse of $a$.
The above $x$ is unique if it exists and is denoted by  $a^{\tiny\textcircled{\tiny \#}}$.
Gao and Chen \cite{GC} introduced the notion of pseudo core inverse as a generalization of the core inverse in $\ast$-rings.
Let $a\in \mathcal{R}$.
If there exists $n\in \mathbb{N}^{+}$ and the unique $x\in \mathcal{R}$ such that the following three equations hold:
\begin{equation}
 xa^{n+1}=a^n,~ax^2=x~\text{and}~(ax)^{\ast}=ax,
\end{equation}
then $a$ is called to be pseudo core invertible, $x$ is the  pseudo core inverse of $a$, denoted by  $a^{\scriptsize\textcircled{\tiny D}}$.
Recently, we defined and investigated right core invertible and right pseudo core invertible elements in $\ast$-rings.
An element $a\in \mathcal{R}$ is right core invertible \cite{WD1} if  there exists some $x\in \mathcal{R}$ such that
\begin{equation}
 axa=a,~ax^2=x~\text{and}~(ax)^{\ast}=ax.
\end{equation}
An element $a\in \mathcal{R}$ is right pseudo core invertible  \cite{WDG1}
if there exists some $x\in \mathcal{R}$ and $n\in \mathbb{N}^{+}$ such that
\begin{equation}
 axa^n=a,~ax^2=x~\text{and}~(ax)^{\ast}=ax.
\end{equation}
It is well known that $a\in \mathcal{R}$ is an EP element if $a \in \mathcal{R}^{\sharp} \cap \mathcal{R}^{\dag}$ and $a^{\sharp}=a^{\dag}$.
It should be pointed out that an EP element  is also known as a *-gMP element, see \cite{PP2004}.
As a matter of convenience, we use the terminology EP element instead of *-gMP element in this paper.
Many authors have published papers on EP elements, see \cite{CV2006,C2012,HK1976,KP2002,MDK2009,XCB1} for example.
In particular, Xu, Chen and Bent\'{i}tez \cite{XCB1} said that $a\in \mathcal{R}$ is an EP element if and only if there exists $x\in \mathcal{R}$ such that
\begin{equation}
xa^2=a,~ax^2=x~\text{and}~(xa)^{\ast}=xa.
\end{equation}
We use the notation $\mathcal{R}^{EP}$ to denote all the EP elements in $\mathcal{R}$.
$a\in \mathcal{R}$ is a *-DMP element \cite{PP2004}
if there exists a positive integer $n$ such that  $(a^{n})^{\sharp}$ and $(a^{n})^{\dag}$ exist with $(a^{n})^{\sharp}=(a^{n})^{\dag}$.
In other words, $a\in \mathcal{R}$ is a *-DMP element  if and only if there exists a positive integer $n$ such that $a^{n}$ is an EP element.
Moreover, $a\in \mathcal{R}$ is a *-DMP element  with index $n$ if and only if $n$ is the smallest positive integer  such that $a^{n}$ is an EP element.
Gao and Chen \cite{GC2018} studied *-DMP elements in $\ast$-semigroups and $\ast$-rings.

Focus on equations (1.1)-(1.5). Let $a\in \mathcal{R}$,
it is of interest to know what is equivalent to the following three systems of equations, respectively.
\begin{equation}
xa^2=a,~ax^2=x~\text{and}~(ax)^{\ast}=xa~\text{for~some}~x\in \mathcal{R}
\end{equation}

\begin{equation}
axa=a,~ax^2=x~\text{and}~(ax)^{\ast}=xa~\text{for~some}~x\in \mathcal{R}
\end{equation}
as well as
\begin{equation}
axa^n=a,~ax^2=x~\text{and}~(ax)^{\ast}=xa~\text{for~some}~x\in \mathcal{R}~\text{and}~n\in \mathbb{N}^{+}
\end{equation}
These are three questions shown in the attached table and we will give answers on these questions.

\section{ Characterizations of EP elements}
In this section, we give new characterizations of EP elements in terms of equations.
We begin with an auxiliary lemma.
\begin{lemma}\label{lem1}
Let $a \in \mathcal{R}$. If there exists some $x\in \mathcal{R}$ such that $a=xa^{2}$ and $(ax)^{\ast}=xa$,
then $xa=ax$.
\end{lemma}
\begin{proof}
Since $a=xa^{2}$ and $(ax)^{\ast}=xa$, we obtain that
\begin{center}
$xa=(ax)^{\ast}=(xa^{2}x)^{\ast}=[(xa)(ax)]^{\ast}=(ax)^{\ast}(xa)^{\ast}=xaax=(xa^{2})x=ax$.
\end{center}
The proof is completed.
\end{proof}

Similarly, we can obtain the following remark.
\begin{remark}\label{rem1}
(1) Let $a \in \mathcal{R}$. If there exists $x\in \mathcal{R}$ such that $x=ax^{2}$ and $(ax)^{\ast}=xa$,
then $xa=ax$.

(2) Let $a \in \mathcal{R}$. If there exists $x\in \mathcal{R}$ such that $a=a^{2}x$ and $(ax)^{\ast}=xa$,
then $xa=ax$.

(3) Let $a \in \mathcal{R}$. If there exists $x\in \mathcal{R}$ such that $x=x^{2}a$ and $(ax)^{\ast}=xa$,
then $xa=ax$.
\end{remark}

Next, we will provide new characterizations of EP elements.
Moreover, from the following Theorem \ref{theo1}, we get the answer of Question 1 (see table below).

\begin{theorem}\label{theo1}
Let $a \in \mathcal{R}$. Then the following statements are equivalent:

(1) $a\in \mathcal{R}^{EP}$.

(2) There exists some $x\in \mathcal{R}$ such that $a=xa^{2}$ and $(ax)^{\ast}=xa$.

(3) There exists some $x\in \mathcal{R}$ such that $a=a^{2}x$ and $(ax)^{\ast}=xa$.

In this case, $a^{\sharp}=a^{\dag}=xax$.
\end{theorem}

\begin{proof}
$(1)\Rightarrow (2)$ It is clear. Indeed, we only have to choose $x=a^{\sharp}=a^{\dag}$.

$(2)\Rightarrow (3)$ By Lemma \ref{lem1}, we have $xa=ax$ and consequently $a=xa^{2}=a^{2}x$.

$(3)\Rightarrow (1)$ Set $y=xax$. By Remark \ref{rem1}, we have $xa=ax$, which gives that $a=axa$, $(ax)^{\ast}=ax$ and $(xa)^{\ast}=xa$.
Then we can obtain that
$aya=axaxa=a$ and $yay=xaxaxax=xax=y$.
As $(ax)^{\ast}=ax$, we have $ay=axax=ax$, which means that $(ay)^{\ast}=ay$.
Moreover, $ya=xaxa=xa$, consequently $(ya)^{\ast}=ya$ and $ya=xa=ax=ay$.
Thus, we can obtain that $a$ is EP, and $y=a^{\sharp}=a^{\dag}$.
\end{proof}

\begin{remark}
The element $x$ in condition (2)(or (3)) of the above Theorem \ref{theo1} is not the group inverse (or the MP inverse) of $a$, in general.
Furtherly, $x$ is not unique.
For example, set $\mathcal{R}=\mathbb{Z}$ and $a=0$. Clearly $a^{\sharp}=a^{\dag}=0$, but we can choose $x=0$ or $x=1$.
\end{remark}

Similar to Theorem \ref{theo1}, it is necessary to consider the relationship between the following condition (1) and condition (2).

Condition (1): $a\in \mathcal{R}^{EP}$.

Condition (2): There exists some $x\in \mathcal{R}$ such that $x=ax^{2}$ and $(ax)^{\ast}=xa$.\\
It is easy to get $(1) \Rightarrow (2)$, but $(2) \nRightarrow (1)$.
For example: Set $\mathcal{R}=\mathbb{Z}$ and $a=2$. Clearly $a\not\in \mathcal{R}^{EP}$, but we can choose $x=0$.
In fact, condition (2) would imply that $x$ (instead of $a$) is a EP element.
Using additional assumption in condition (2), we will prove the equivalence in the following.

\begin{proposition}\label{prop1}
Let $a \in \mathcal{R}$. Then the following statements are equivalent:

(1) $a\in \mathcal{R}^{EP}$.

(2) There exists the unique $x\in \mathcal{R}$ such that $axa=a$, $x=ax^{2}$ and $(ax)^{\ast}=xa$.

In this case, $a^{\sharp}=a^{\dag}=x$.
\end{proposition}

\begin{proof}
$(1)\Rightarrow (2)$ It is clear by letting $x=a^{\sharp}=a^{\dag}$.
Next, we prove the unique of $x$.
Suppose that there exist $x,~y$ such that $axa=a$, $x=ax^{2}$, $(ax)^{\ast}=xa$
and  $aya=a$, $y=ay^{2}$ and $(ay)^{\ast}=ya$.
Then $ax=xa$ and $ay=ya$ by Remark~\ref{rem1}.
Thus $x=ax^2=ayax^2=yax$ and similarly $y=xay$.
Hence $x=yax=y(xa)^{\ast}=y(xaya)^{\ast}=y(ya)^{\ast}=yay=ay^2=y.$

$(2)\Rightarrow (1)$ By Remark \ref{rem1}, one can see that $ax=xa$.
Note that $axa=a$, then we have $a=xa^{2}$, and consequently $a\in \mathcal{R}^{EP}$ by Theorem \ref{theo1}.

As $x=ax^{2}$, we have $ax=a^{2}x^{2}$ and consequently $x=a^{\sharp}a^{2}x^{2}=a^{\sharp}ax=(a^{\sharp})^{2}a^{2}x=(a^{\sharp})^{2}a=a^{\sharp}$.
\end{proof}

From Proposition \ref{prop1}, we get the answer of Question 2 (see table below).
Moreover, motived by Theorem \ref{theo1},
it is necessary to consider wether we can delete the condition of $x=ax^{2}$ in Proposition \ref{prop1}.
The problem is provided as follows:
\begin{center}
there exists some $x\in \mathcal{R}$ such that $axa=a$ and $(ax)^{\ast}=xa$ $\overset{?}{\Longrightarrow}$ $a\in \mathcal{R}^{EP}$.
\end{center}
Further, if there exists some $x\in R$ such that $axa=a$ and $(ax)^{\ast}=xa$, then one can see $aya=a$, $yay=y$ and $(ay)^{\ast}=ya$, where $y=xax$.
In fact, this question have a more detailed expression:
\begin{center}
there exists some $x\in R$ such that $axa=a$, $xax=x$ and $(ax)^{\ast}=xa$ $\overset{?}{\Longrightarrow}$ $a\in \mathcal{R}^{EP}$?
\end{center}

\begin{example}\label{eg1}
Let $\mathcal{R} = M_{2}(\mathbb{C})$, and set the involution of $\mathcal{R}$ as the transpose of matrices.
Take $a=\left(
          \begin{array}{cc}
            1 & i \\
            i & -1 \\
          \end{array}
        \right)$ and
$x=\left(
     \begin{array}{cc}
       2 & 0 \\
       0 & 1 \\
     \end{array}
   \right)$.
It is not difficult to work out that
$ax=(xa)^{\ast}$ and $axa=a$.
But we can check that $a^{\ast}a=O_{2}$, which implies that $a$ is not EP.
\end{example}

Similarly to the proof of Proposition~\ref{prop1}, we can get the following result.
\begin{proposition}
Let $a \in \mathcal{R}$. Then the following statements are equivalent:

(1) $a\in \mathcal{R}^{EP}$.

(2) There exists the unique $x\in \mathcal{R}$ such that $axa=a$, $x=x^{2}a$ and $(ax)^{\ast}=xa$.

(3) There exists the unique $x\in \mathcal{R}$ such that $a^2x=a$, $x=x^{2}a$ and $(ax)^{\ast}=xa$.

(4) There exists the unique $x\in \mathcal{R}$ such that $xa^2=a$, $x=x^{2}a$ and $(ax)^{\ast}=xa$.

(5) There exists the unique $x\in \mathcal{R}$ such that $a^2x=a$, $x=ax^{2}$ and $(ax)^{\ast}=xa$.

(6) There exists the unique $x\in \mathcal{R}$ such that $xa^2=a$, $x=ax^{2}$ and $(ax)^{\ast}=xa$.

(7) There exists the unique $x\in \mathcal{R}$ such that $a^2x=a$, $x=xax$ and $(ax)^{\ast}=xa$.

(8) There exists the unique $x\in \mathcal{R}$ such that $xa^2=a$, $x=xax$ and $(ax)^{\ast}=xa$.

In this case, $a^{\sharp}=a^{\dag}=x$.

\end{proposition}

From Example~\ref{eg1}, we know that ``$axa=a$ and $(ax)^{\ast}=xa$'' may not imply that $a$ is an EP element, however,  $a$ will be an EP element if we add an extra condition.

\begin{lemma}\label{lem2}\emph{\cite{KP2002}}
(1) Let $\mathcal{R}$ be an associative ring with unit. Then $a\in \mathcal{R}^{\sharp}$ is equivalent to $a\in a^{2}\mathcal{R} \cap \mathcal{R}a^{2}$.

(2) Let $\mathcal{R}$ be a $\ast$-ring and $a\in \mathcal{R}^{\sharp}$. Then $a\in \mathcal{R}^{EP}$ is equivalent to $(a^{\sharp}a)^{\ast}=a^{\sharp}a$.
\end{lemma}

\begin{theorem}\label{theo2}
Let $a\in \mathcal{R}^{\{1,3\}}$. Then the following statements are equivalent:

(1) $a\in \mathcal{R}^{EP}$;

(2) There exists some $x\in \mathcal{R}$ such that $a=axa$ and $(ax)^{\ast}=xa$.
\end{theorem}

\begin{proof}
We only have to prove $(2)\Rightarrow (1)$.
As $(ax)^{\ast}=xa$, it follows that
$xa=(ax)^{\ast}=(aa^{(1,3)}ax)^{\ast}=(ax)^{\ast}(aa^{(1,3)})^{\ast}=xaaa^{(1,3)}$.
Multiplying on the left by $a$ gives $a=a^{2}a^{(1,3)}\in a^{2}\mathcal{R}$.
Moreover, $aa^{(1,3)}=axaa^{(1,3)}=(axaa^{(1,3)})^{\ast}=aa^{(1,3)}(ax)^{\ast}=aa^{(1,3)}xa$.
Multiplying on the right by $a$ gives $a=aa^{(1,3)}xa^{2}\in \mathcal{R}a^{2}$.
So,  by Lemma \ref{lem2},
it implies that $a\in a^{2}\mathcal{R} \cap \mathcal{R}a^{2}$ and $a\in \mathcal{R}^{\sharp}$.
Note that $a=a^{2}a^{(1,3)}$. Then this shows $a^{\sharp}a=a^{\sharp}a^{2}a^{(1,3)}=aa^{(1,3)}$,
which implies that $(a^{\sharp}a)^{\ast}=a^{\sharp}a$, and,  by Lemma \ref{lem2}, $a\in \mathcal{R}^{EP}$.
\end{proof}

Similarly, we can obtain the following results whenever $a\in \mathcal{R}^{\{1,4\}}$.

\begin{proposition}\label{prop2}
Let $a\in \mathcal{R}^{\{1,4\}}$. Then the following statements are equivalent:

(1) $a\in \mathcal{R}^{EP}$;

(2) There exists some $x\in \mathcal{R}$ such that $a=axa$ and $(ax)^{\ast}=xa$.
\end{proposition}

It is well known that $a\in \mathcal{R}^{\dag}$ if and only if $a\in \mathcal{R}^{\{1,3\}} \cap \mathcal{R}^{\{1,4\}}$.
By Theorem \ref{theo2} and Proposition \ref{prop2}, we can obtain the following result.

\begin{corollary}\label{cor1}
Let $a\in \mathcal{R}^{\dag}$. Then the following statements are equivalent:

(1) $a\in \mathcal{R}^{EP}$;

(2) There exists some $x\in \mathcal{R}$ such that $a=axa$ and $(ax)^{\ast}=xa$.
\end{corollary}

Next, we will consider whether we can change the condition $a\in \mathcal{R}^{\dag}$
into $a\in \mathcal{R}^{\sharp}$ in Corollary \ref{cor1}.

\begin{theorem}\label{theo3}
Let $a\in a^{2}\mathcal{R}$. Then the following statements are equivalent:

(1) $a\in \mathcal{R}^{EP}$;

(2) There exists some $x\in \mathcal{R}$ such that $a=axa$ and $(ax)^{\ast}=xa$.
\end{theorem}
\begin{proof}
As $a\in a^{2}\mathcal{R}$, then we have $a=a^{2}t$ for some $t\in \mathcal{R}$.
By $ax=(xa)^{\ast}$, it implies that $ax=(xa)^{\ast}=(xa^{2}t)^{\ast}=(at)^{\ast}(xa)^{\ast}=t^{\ast}a^{\ast}ax$.
Multiplying on the right by $a$ gives $a=t^{\ast}a^{\ast}a$.
Then we can obtain that $a^{\ast}=a^{\ast}at$, and consequently
$a=t^{\ast}a^{\ast}a=t^{\ast}(a^{\ast}at)a=ata$.
Moreover, it follows from $a=t^{\ast}a^{\ast}a$ that $at=(t^{\ast}a^{\ast}a)t=(at)^{\ast}at$,
which gives that $(at)^{\ast}=at$.
Hence, one can see that $a\in \mathcal{R}^{\{1,3\}}$.
This implies that $a\in \mathcal{R}^{EP}$ by Theorem \ref{theo2}.
\end{proof}

Similarly, by Proposition \ref{prop2}, we can obtain the following results whenever $a\in \mathcal{R}a^{2}$.

\begin{proposition}\label{prop3}
Let $a\in \mathcal{R}a^{2}$. Then the following statements are equivalent:

(1) $a\in \mathcal{R}^{EP}$;

(2) There exists some $x\in \mathcal{R}$ such that $a=axa$ and $(ax)^{\ast}=xa$.
\end{proposition}

By Theorem \ref{theo3} and Proposition \ref{prop3}, we can obtain the following result.

\begin{corollary}\label{cor2}
Let $a\in \mathcal{R}^{\sharp}$. Then the following statements are equivalent:

(1) $a\in \mathcal{R}^{EP}$;

(2) There exists some $x\in \mathcal{R}$ such that $a=axa$ and $(ax)^{\ast}=xa$.
\end{corollary}

\section{ Characterizations of central EP elements}

\begin{definition}\label{def1}
An element $a \in \mathcal{R}$ is said to be central EP (Abb. CEP), if
there exists $x \in \mathcal{R}$ such that
\begin{center}
$axa = a$, $(ax)^{\ast}=xa\in C(\mathcal{R})$,
\end{center}
where $C(\mathcal{R})$ is the center of $\mathcal{R}$.
\end{definition}

\begin{remark}
In fact, if $a$ is CEP, by $xa \in C(\mathcal{R})$ and $axa=a$, then we have $a=a(xa)=xa^{2}$.
Thus, by Lemma \ref{lem1} and Theorem \ref{theo1}, we obtain that $ax=xa$, and $a$ is also EP.
\end{remark}

\begin{example}
Let $\mathcal{R} = M_{2}(\mathbb{Z})$, and set the involution of $\mathcal{R}$ as the transpose of matrices.
It is not difficult to work out that $a^{\sharp}=a^{\dag}=a^{\ast}=a$
by taking $a=\left(
          \begin{array}{cc}
            1 & 0 \\
            0 & 0 \\
          \end{array}
        \right)$.
This implies that $a$ is EP.
Moreover, it is known that the center of $M_{2}(\mathbb{Z})$ contain only $I_{2}$ and $O_{2}$.
However, we can not find the element $x\in \mathcal{R}$ satisfying the definition of CEP, such that $ax$ is $I_{2}$ or $O_{2}$.
This shows that $a$ is not CEP.
And it illustrates CEP is a proper subset of EP.
\end{example}

In paper \cite{WZ1}, the authors introduced the definition of central group inverse and studied the related properties.

\begin{definition}
An element $a \in \mathcal{R}$ is said to be central group invertible \cite{WZ1}
if there exists $x \in \mathcal{R}$ satisfying
$xa \in C(\mathcal{R})$, $xax= x$, and $a^{2}x = a$.
\end{definition}
In this case, $x$ is called a central group inverse of $a$ (denoted by $a^{\copyright}$  ).

\begin{theorem}\label{theo4}
An element $a\in \mathcal{R}$ is CEP if and only if $a$ is central group invertible and MP invertible, and $a^{\copyright}=a^{\dag}$.
\end{theorem}
\begin{proof}
As $a$ is CEP, there exists some $x\in \mathcal{R}$ such that $axa = a$, $(ax)^{\ast}=xa$ and $xa \in C(\mathcal{R})$.
Set $z=xax$. It is not difficult to check that $aza=a$, $zaz=z$ and $az=za=xa=ax$,
which gives that $a$ is central group invertible and MP invertible, and $a^{\copyright}=a^{\dag}=xax$.
Conversely, we only have to choose $y=a^{\copyright}=a^{\dag}$.
Then we can check that $aya = a$, $(ay)^{\ast}=ya$ and $ya \in C(\mathcal{R})$,
and consequently $a$ is CEP.
\end{proof}

\begin{remark}\label{rem2}

If $a$ is CEP, by the definition, there exists $x\in \mathcal{R}$ such that $axa=a$ and $(ax)^{\ast}=xa\in C(\mathcal{R})$.
In this case, the element $x$ is not unique, in generally.
For example,  set $\mathcal{R}=\mathbb{Z}$ and $a=0$. Clearly, the element $x$ can be arbitrary.
\end{remark}

\begin{proposition}\label{pro311}
Let $a\in \mathcal{R}$. Then the following statements are equivalent:

(1) $a$ is CEP.

(2) there exists the unique $z \in \mathcal{R}$ such that $aza = a$, $zaz=z$ and $(az)^{\ast}=za\in C(\mathcal{R})$.

In this case, $z$ is called a CEP-inverse of $a$ (denoted by $a^{\copyright\dag}$  ).
\end{proposition}
\begin{proof}
As $a$ is CEP, there exists some $x\in \mathcal{R}$ such that $axa = a$, $(ax)^{\ast}=xa \in C(\mathcal{R})$.
Set $z=xax$. It is not difficult to check that $aza=a$, $zaz=z$ and $(ax)^{\ast}=xa \in C(\mathcal{R})$.
In order to prove the uniqueness of $z$, assume that there exists $y$ satisfying the condition (2).
Then, by $ay\in C(\mathcal{R})$, we have $z=zaz=z(aya)z=zaz(ay)=zay$.
Similarly, by $za\in C(\mathcal{R})$, one can see $y=yay=y(aza)y=(za)yay=zay$, which implies that $z=y$.
\end{proof}

\begin{remark}\label{rem8}
If $a$ is CEP, then $a^{\copyright\dag}=a^{\copyright}=a^{\sharp}=a^{\dag}$
\end{remark}

\begin{proposition}\label{prop4}
An element $a \in \mathcal{R}$ is CEP if and only if $a$ is
EP and the idempotent $1-a^{\dag}a$ is central.
\end{proposition}
\begin{proof}
By Remark \ref{rem8}, we have $1-a^{\dag}a=1-a^{\copyright\dag}a$, which gives $1-a^{\dag}a$ is central.
Conversely, we can choose $a^{\dag}$ as the CEP inverse of $a$.
\end{proof}

Next, we will give the definition of central projection.
An element $p\in \mathcal{R}$ is projection if $p=p^{2}=p^{\ast}$,
and $p$ is said to be a central projection if $p=p^{2}=p^{\ast}\in C(\mathcal{R})$.
We use the symbols $P(\mathcal{R})$ and $CP(\mathcal{R})$
to denote the set of all projections and the set of all central projections in $\mathcal{R}$, respectively.
By Proposition \ref{prop4}, we can obtain the following corollary.

\begin{corollary}
Every EP element of $\mathcal{R}$ is CEP if and only if $P(\mathcal{R})=CP(\mathcal{R})$.
\end{corollary}

Recall that an element $a$ of $\mathcal{R}$ is called clean
if it is the sum of an idempotent $e \in \mathcal{R}$ and a unit $u \in \mathcal{R}$.
Such a clean decomposition $a = e + u$ in a ring $\mathcal{R}$ is called strongly clean if $eu = ue$.
In what follows, the clean decompositions for CEP elements are considered.

\begin{proposition}\label{prop5}
Let $a\in \mathcal{R}$. Then the following statements are equivalent:

(1) $a$ is CEP.

(2) There exist $u\in U(\mathcal{R})$ and $p\in CP(\mathcal{R})$ such that $a = u + p$ and $up=-p$.

In addition, $u$ and $p$ are unique.
\end{proposition}
\begin{proof} $(1)\Rightarrow(2)$.
As $a$ is CEP, by Theorem \ref{theo4}, it follows that $a$ is central group invertible and MP invertible.
Moreover, $a^{\copyright}=a^{\sharp}=a^{\dag}$.
It is not difficult to check that $u=a-1+aa^{\dag}$ is invertible and $u^{-1}=a^{\dag}-1+aa^{\dag}$.
Write $p=1-aa^{\dag}$. Then, by Proposition \ref{prop4}, one can see that $p\in CP(\mathcal{R})$, $up=-p$ and $a=u+p$.

$(2)\Rightarrow(1)$. If there exist $u\in U(\mathcal{R})$ and $p\in CP(\mathcal{R})$ such that $a = u + p$ and $up=-p$,
then, for $x=u^{-1}(1-p)$, one can see that
$$ax=xa=u^{-1}(1-p)(u+p)=u^{-1}(1-p)u=u^{-1}u(1-p)=1-p\in C(\mathcal{R}).$$
This leads to $axa=(u+p)(1-p)=u-up=u-(-p)=a$ and $(ax)^*=1-p=xa$.
Hence, $a$ is CEP.

In order to prove that $u$ and $p$ are unique,
assume that there exist $u,v\in U(\mathcal{R})$ and $p,q\in CP(\mathcal{R})$
such that $a = u + p=v+q$, $up=-p$ and $vq=-q$.
Notice that $aq=(v+q)q=vq+q=-q+q=0$ and $p=-u^{-1}p=-u^{-1}(a-u)=-u^{-1}a+1$,
which imply $(1-p)q=u^{-1}aq=0$.
So, $q=pq$ and similarly $p=qp$.
Now, $p=qp=pq=q$ and then $a=u+p=v+p$ gives $u=v$.
\end{proof}

\begin{proposition} \label{prop6}
Let $a\in \mathcal{R}$. Then the following statements are equivalent:

(1) $a$ is CEP.

(2) There exist $u\in U(\mathcal{R})$ and $q\in CP(\mathcal{R})$ such that $a=uq$.

In addition, $q$ is unique.
\end{proposition}

\begin{proof}
$(1)\Rightarrow(2)$. Suppose that $a$ is CEP.
Let $u = a-1+a^{\dag}a$ and $q = a^{\dag}a$.
Then, by Proposition \ref{prop5},
we have $u \in U(\mathcal{R})$ and $q \in CP(\mathcal{R})$.
Moreover, it is easy to check that $uq = a$.

$(2) \Rightarrow (1)$.
Let $a = uq$ with $u \in U(\mathcal{R})$ and $q\in CP(\mathcal{R})$.
Set $x = u^{-1}q$.
It can be easily checked that $x$ is just the CEP-inverse of $a$.

To show that $q$ is unique, let there exist $u, v\in U(\mathcal{R})$ and $p, q\in CP(\mathcal{R})$ such that $a=uq=vp$.
Since $aq=uq=a$ and $p=v^{-1}a$, we get $p(1-q)=v^{-1}a(1-q)=0$ which yields $p=pq$.
In a same manner, we have that $q=qp$ and thus $p=pq=qp=q$.
\end{proof}

\begin{remark}
The element $u\in U(\mathcal{R})$ of condition (2) in Proposition \ref{prop6} is not unique, in general.
For instance, set $\mathcal{R}=\mathbb{Z}$ and $a=0$. Clearly, $q=0$, and we can chose $u=1$ or $u=-1$.
\end{remark}

Using tripotents, we characterize CEP elements in the following results.

\begin{proposition}
Let $a\in \mathcal{R}$. Then the following statements are equivalent:

(1) $a$ is CEP.

(2) There exist $u\in U(\mathcal{R})$ and $p=p^3\in \mathcal{R}$ such that $a = u + p$, $p^2\in CP(\mathcal{R})$ and $up^2=-p$.
\end{proposition}
\begin{proof} $(1)\Rightarrow(2)$.
Applying Proposition \ref{prop5},
there exist $u\in U(\mathcal{R})$ and $p\in CP(\mathcal{R})$ such that $a=u + p$ and $up=-p$.
Therefore, $p=p^3\in \mathcal{R}$, $p^2=p\in CP(\mathcal{R})$ and $up^2=-p$.

$(2)\Rightarrow(1)$. Suppose that there exist $u\in U(\mathcal{R})$ and $p=p^3\in \mathcal{R}$ such that $a = u + p$, $p^2\in CP(\mathcal{R})$ and $up^2=-p$.
Let $x=u^{-1}(1-p^2)$. Then one can see that
\begin{center}
$xa=u^{-1}(1-p^2)(u+p)=u^{-1}(1-p^2)u=u^{-1}u(1-p^2)=1-p^2\in C(\mathcal{R})$,
\end{center}
and $ax=(u+p)u^{-1}(1-p^2)=(1-p^2)(u+p)u^{-1}=(1-p^2)uu^{-1}=1-p^2=xa$.
Moreover, we obtain $axa=(u+p)(1-p^2)=u-up^2=u-(-p)=a$ and $(ax)^*=1-p^2=xa$ give that $x$ is a CEP-inverse of $a$.
\end{proof}

\begin{proposition}
Let $a\in \mathcal{R}$. Then the following statements are equivalent:

(1) $a$ is CEP.

(2) There exist $u\in U(\mathcal{R})$ and $q=q^3\in \mathcal{R}$ such that $a=uq^2$ and $q^2\in CP(\mathcal{R})$.
\end{proposition}

\begin{proof}
$(1)\Rightarrow(2)$. By Proposition \ref{prop6},
there exist $u\in U(\mathcal{R})$ and $q\in CP(\mathcal{R})$ such that $a=uq$.
So, $q=q^3\in \mathcal{R}$, $q^2=q\in CP(\mathcal{R})$ and $a=uq^2$.

$(2) \Rightarrow (1)$. If $x = u^{-1}q^2$, we can verify that $x$ is a CEP-inverse of $a$.
\end{proof}

Next, the CEP properties of the sum and product of two CEP elements are considered.

\begin{proposition}\label{prop-product-CEP}
If $a, b\in \mathcal{R}$ are CEP, then $ab$ is CEP.
In addition, $b^{\copyright\dag}a^{\copyright\dag}$ is the CEP-inverse of $ab$.
\end{proposition}

\begin{proof}
Assume that $x\in \mathcal{R}$ is a CEP-inverse of $a$ and $y\in \mathcal{R}$ is a CEP-inverse of $b$.
Since $xa, ax, yb$ and $by\in C(\mathcal{R})$,
one can see that
\begin{center}
$yxab=ybxa\in C(\mathcal{R})$, \\
$abyxab=axabyb=ab$, $yxabyx=ybyxax=yx$\\
and $(abyx)^*=(axby)^*=(by)^*(ax)^*=ybxa=yxab$.
\end{center}
This gives that $yx$ is a CEP-inverse of $ab$.
\end{proof}

By Proposition \ref{prop-product-CEP}, we obtain the next consequence.

\begin{corollary}
If $a\in \mathcal{R}$ is CEP and $n\in \mathbb{N}^{+}$,
then $a^n$ is CEP.
In addition, $(a^{\copyright\dag})^n$ is the CEP-inverse of $a^n$.
\end{corollary}

\begin{proposition}
Let $a, b\in \mathcal{R}$ be CEP,
$x\in \mathcal{R}$ be the CEP-inverse of $a$ and $y\in \mathcal{R}$ be the CEP-inverse of $b$.
If $xb=bx=0=ya=ay$, then $a+b$ is CEP.
In addition,
$x+y$ is the CEP-inverse of $a+b$.
\end{proposition}

\begin{proof}
According to the definition of the CEP-inverse, we can check that $x+y$ is the CEP-inverse of $a$ directly.
\end{proof}

For an idempotent $p\in \mathcal{R}$, we can represent arbitrary element $a\in \mathcal{R}$ as
$$a=\sbmatrix{cc}
a_{11}&a_{12}\\a_{21}&a_{22}\endsbmatrix_p,$$
where $a_{11}=pap$, $a_{12}=pa(1-p)$, $a_{21} = (1-p)ap$, $a_{22}= (1-p)a(1-p)$.

If $p=p^{2}=p^*$, then
$$a^*=\sbmatrix{ccc}a_{11}^*&a_{21}^*\\a_{12}^*&a_{22}^*\endsbmatrix_{p}.$$

Now, we give some characterizations of CEP elements in the following theorem.

\begin{theorem}\label{th-CEP-char}
Let $a\in \mathcal{R}$. Then the following statements are equivalent:

(1) $a$ is CEP.

(2) There exist $p\in CP(\mathcal{R})$ such that
$$a=\sbmatrix{cc}
a_{1}&0\\0&0\endsbmatrix_p,$$
where $a_1\in(p\mathcal{R}p)^{-1}$.

(3) There exist $p\in CP(\mathcal{R})$ such that $a\in(p\mathcal{R}p)^{-1}$.

In addition, $x=a^{-1}_{p\mathcal{R}p}$ is the CEP-inverse of $a$.
\end{theorem}

\begin{proof}
$(1)\Rightarrow(2)\wedge(3)$. If $x\in \mathcal{R}$ is the CEP-inverse of $a$ and $p=ax=xa$,
then $p\in CP(\mathcal{R})$.
Moreover, $a$ has representation as in (2) and $a_1=pap=ap=axa=a$.
We deduce that $a_1=a\in(p\mathcal{R}p)^{-1}$ and $x$ is the inverse of $a=a_1$ in $p\mathcal{R}p$.

$(2) \Rightarrow (1)$.
For $x=\sbmatrix{cc}
a_{1}^{-1}&0\\0&0\endsbmatrix_p$,
we easily check that $x$ is the CEP-inverse of $a$.

$(3) \Rightarrow (1)$. Because $a\in(p\mathcal{R}p)^{-1}$,
then $a\in p\mathcal{R}p$ gives $a=pa=ap$.
Let $x=a^{-1}_{p\mathcal{R}p}$, where $a^{-1}_{p\mathcal{R}p}$ is the inverse of $a$ in $p\mathcal{R}p$.
By $xp=a^{-1}_{p\mathcal{R}p}=px$, we obtain $xa=xpa=p=ax$ and $axa=pa=a$.
From $p\in CP(\mathcal{R})$, we have that $x$ is the CEP-inverse of $a$.
\end{proof}

Using CEP elements, we define and study the CEP relation in a ring with
involution.

\begin{definition}
Let $a\in \mathcal{R}$ be CEP and $b\in \mathcal{R}$.
Then $a$ is below $b$ under the CEP relation $($denoted by $a\leq^{CEP}b)$, if
$a=a^{\copyright\dag}ab$.
\end{definition}

\begin{proposition}\label{prop321}
Let $a\in \mathcal{R}$ be CEP and $b\in \mathcal{R}$. Then the following statements are equivalent:

(1) $a\leq^{CEP}b$.

(2) $a^{\copyright\dag}a=a^{\copyright\dag}b$.

(3) $aa^{\copyright\dag}=ba^{\copyright\dag}$.

(4) There exist $p\in CP(\mathcal{R})$ such that
$$a=\sbmatrix{cc}
a_{1}&0\\0&0\endsbmatrix_p\quad and \quad b=\sbmatrix{cc}
a_{1}&0\\0&b_{2}\endsbmatrix_p,$$
where $a_1\in(p\mathcal{R}p)^{-1}$.
\end{proposition}

\begin{proof}
$(1) \Rightarrow (2)$. As $a=a^{\copyright\dag}ab$, it follows that $a^{\copyright\dag}a=a^{\copyright\dag}a^{\copyright\dag}ab$.
Since $a^{\copyright\dag}a\in C(\mathcal{R})$ and $a^{\copyright\dag}aa^{\copyright\dag}=a^{\copyright\dag}$,
we have $a^{\copyright\dag}a=a^{\copyright\dag}(a^{\copyright\dag}a)b=(a^{\copyright\dag}a)a^{\copyright\dag}b=a^{\copyright\dag}b$.

$(2) \Rightarrow (3)$. Since $a^{\copyright\dag}a=a^{\copyright\dag}b$, multiplying on the left by $a$,
we obtain $aa^{\copyright\dag}a=aa^{\copyright\dag}b$, which gives $a=aa^{\copyright\dag}b$.
By $aa^{\copyright\dag}\in C(\mathcal{R})$, it follows that $a=baa^{\copyright\dag}$.
Multiplying on the right by $a^{\copyright\dag}$,
it implies that $aa^{\copyright\dag}=b(aa^{\copyright\dag})a^{\copyright\dag}=ba^{\copyright\dag}$.

$(3) \Rightarrow (1)$. As $aa^{\copyright\dag}=ba^{\copyright\dag}$, multiplying on the right by $a$,
we obtain $a=ba^{\copyright\dag}a$.
Hence, it follows from $a^{\copyright\dag}a \in \mathcal{R}$ that $a=a^{\copyright\dag}ab$.

$(1) \Leftrightarrow (4)$. It follows by Theorem \ref{th-CEP-char}.
\end{proof}

We now prove that the CEP relation is a partial order on the set of all CEP elements of $\mathcal{R}$.

\begin{theorem} The relation $\leq^{CEP}$ is a partial order
on the set $\{r\in \mathcal{R}: r\ is\ CEP \}$.
\end{theorem}

\begin{proof} Obviously, the relation $\leq^{CEP}$ is
reflexive.

To prove that $\leq^{CEP}$ is
antisymmetric, assume that $a\leq^{CEP}b$ and $b\leq^{CEP}a$.
Then we have $a=a^{\copyright\dag}ab$ and $b=b^{\copyright\dag}ba$.
Note that $a^{\copyright\dag}a\in C(\mathcal{R})$, it gives that
\begin{center}
$a=a^{\copyright\dag}ab=a^{\copyright\dag}a(b^{\copyright\dag}ba)=b^{\copyright\dag}baa^{\copyright\dag}a=b^{\copyright\dag}ba=b$.
\end{center}
In order to show that $\leq^{CEP}$ is transitive, assume that $a\leq^{CEP}b$ and $b\leq^{CEP}c$.
From  Proposition \ref{prop321}, it follows that $a^{\copyright\dag}a=a^{\copyright\dag}b$ and $b^{\copyright\dag}b=b^{\copyright\dag}c$.
This implies that $a^{\copyright\dag}c=a^{\copyright\dag}a^{\copyright\dag}ac
=a^{\copyright\dag}a^{\copyright\dag}bc=a^{\copyright\dag}a^{\copyright\dag}bbb^{\copyright\dag}c=a^{\copyright\dag}a^{\copyright\dag}bbb^{\copyright\dag}b
=a^{\copyright\dag}a^{\copyright\dag}bb=a^{\copyright\dag}a^{\copyright\dag}ab=a^{\copyright\dag}aa^{\copyright\dag}a=a^{\copyright\dag}a$.
Therefore, by Proposition \ref{prop321}, it gives that $a\leq^{CEP}c$.
\end{proof}

\section{ Characterizations *-DMP elements}
In this section, we focus on answering Question 3. That is, what is the system (1.8) going to be equivalent?
Firstly, the *-DMP element was considered. The notion of *-DMP element was introduced by Patr\'{i}cio and Puystjens~\cite{PP2004} in $\ast$-rings.
An element $a\in \mathcal{R}$ is said to be a *-DMP element if there exists $n\in \mathbb{N}^{+}$ such that $(a^n)^{\dag}$ and  $(a^n)^{\sharp}$ exist and are equal.

\begin{lemma}\cite{GC2018,PP2004}\label{4} Let $a\in \mathcal{R}$. Then the following are equivalent:

(1) $a$ is a *-DMP element.

(2) $aa^D$ is Hermitian.
\end{lemma}

\begin{theorem}\label{theo5}
Let $a\in \mathcal{R}$. Then the following are equivalent:

(1) $a$ is a *-DMP element.

(2) There exists $n\in \mathbb{N}^{+}$ and the unique $x \in \mathcal{R}$  such that
\begin{center}
$axa^n = a^n$, $ax^2=x$ and $(ax)^{\ast}=xa$.
\end{center}
\end{theorem}

\begin{proof}  $(1)\Rightarrow (2)$
It is clear by letting $x=a^D$ and Lemma \ref{4}.
Moreover, the element $x$ satisfying the condition must be $a^{D}$.
So, the uniqueness can be proved.

$(2)\Rightarrow (1)$
As $ax^2=x$ and $(ax)^{\ast}=xa$, it follows from Lemma \ref{lem1} that $ax=xa$.
Thus, it implies that $x=a^D$ and $aa^D$ is Hermitian.
On account of Lemma \ref{4}, we conclude that $a$ is a *-DMP element.
\end{proof}

By Theorem \ref{theo5}, the answer of the question 3 is given (see table below).
Next, we investigate the clean decompositions for *-DMP elements.

\begin{proposition}\label{prop4.3}
Let $a\in \mathcal{R}$ and ${\rm ind}(a)=n$. Then the following statements are equivalent:

(1) $a$ is a *-DMP element.

(2) There exist $u\in U(\mathcal{R})$ and $p\in P(\mathcal{R})$ such that $a^n=u+p$, $pa=ap$ and $up=pu=-p$.

In addition, $u$ and $p$ are unique.
\end{proposition}
\begin{proof} $(1)\Rightarrow(2)$. Assume that
there exists $x\in \mathcal{R}$ such that
$axa^n = a^n$, $ax^2=x$ and $(ax)^{\ast}=xa$. Let $u=a^n-1+ax$ and $p=1-ax$.
We can verify that $u$ is invertible and $u^{-1}=x^n-1+ax$.
Also, we observe that $a^n=u+p$, $up=pu=-p$, $pa=ap$ and $p=p^2=p^*$.

$(2)\Rightarrow(1)$. Let there exist $u\in U(\mathcal{R})$ and $p\in P(\mathcal{R})$ such that $a^n=u+p$, $pa=ap$ and $up=-p$.
Since $a$ commutes with $p$ and $a^n=u+p$, we deduce that $a$ commutes with $u$.
Set $x=u^{-1}(1-p)a^{n-1}$. Then
$xa=u^{-1}(1-p)a^{n}=u^{-1}(1-p)(u+p)=u^{-1}(1-p)u=u^{-1}u(1-p)=1-p\in P(R),$ $(ax)^*=(xa)^*=xa=1-p$ and
$axa^n=(1-p)(u+p)=u-pu=u-(-p)=a^n$.
By Theorem \ref{theo5}, $a$ is a *-DMP element.

To verify that $u$ and $p$ are unique, suppose that there exist
$u,v\in U(\mathcal{R})$ and $p,q\in P(\mathcal{R})$ such that $a^n= u + p=v+q$, $pa=ap$, $qa=aq$, $up=pu=-p$ and $vq=qv=-q$.
Because $a^nq=(v+q)q=vq+q=-q+q=0$ and $p=-u^{-1}p=-u^{-1}(a^n-u)=-u^{-1}a^n+1$, we obtain $(1-p)q=u^{-1}aq^n=0$ and thus $q=pq$. In a similar way, we get
$p=qp$. Hence, $p=qp=q^*p^*=(pq)^*=q^*=q$ and $u=a^n-p=a^n-q=v$.
\end{proof}

\begin{proposition}
Let $a\in \mathcal{R}$ and ${\rm ind}(a)=n$. Then the following statements are equivalent:

(1) $a$ is a *-DMP element.

(2) There exist $u\in U(\mathcal{R})$ and $q\in P(\mathcal{R})$ such that $a^n=uq=qu$ and $aq=qa$.

In addition, $q$ is unique.
\end{proposition}

\begin{proof}
$(1)\Rightarrow(2)$. If there exists $x\in \mathcal{R}$ such that
$axa^n = a^n$, $ax^2=x$ and $(ax)^{\ast}=xa$, we denote by $u=a^n-1+ax$ and $p=1-ax$.
Then, by Proposition \ref{prop4.3}, we have $u\in U(\mathcal{R})$ and $q\in P(\mathcal{R})$.
Also, we get that $a^n=uq=qu$ and $aq=qa$.

$(2) \Rightarrow (1)$. Suppose that there exist $u\in U(\mathcal{R})$ and $q\in P(\mathcal{R})$ such that $a^n=uq=qu$ and $aq=qa$.
Let $x = a^{n-1}u^{-1}q$. Then $ax=a^{n}u^{-1}q=quu^{-1}q=q=xa$ and $axa^n=q^u=a^n$.
By Theorem \ref{theo5}, one can see $a$ is a *-DMP element.

In order to prove that $q$ is unique, assume that there exist $u,v\in U(\mathcal{R})$ and $p,q\in P(\mathcal{R})$ such that $aq=qa$, $ap=pa$ and $a^n=qu=uq=vp=pv$.
From $a^nq=uq=a^n$ and $p=v^{-1}a^n$, we have $p(1-q)=v^{-1}a^n(1-q)=0$. Hence, $p=pq$ and similarly
$q=qp$. Now, we obtain $p=pq=p^*q^*=(qp)^*=q^*=q$.
\end{proof}

\newpage

\includegraphics[width=6.5in]{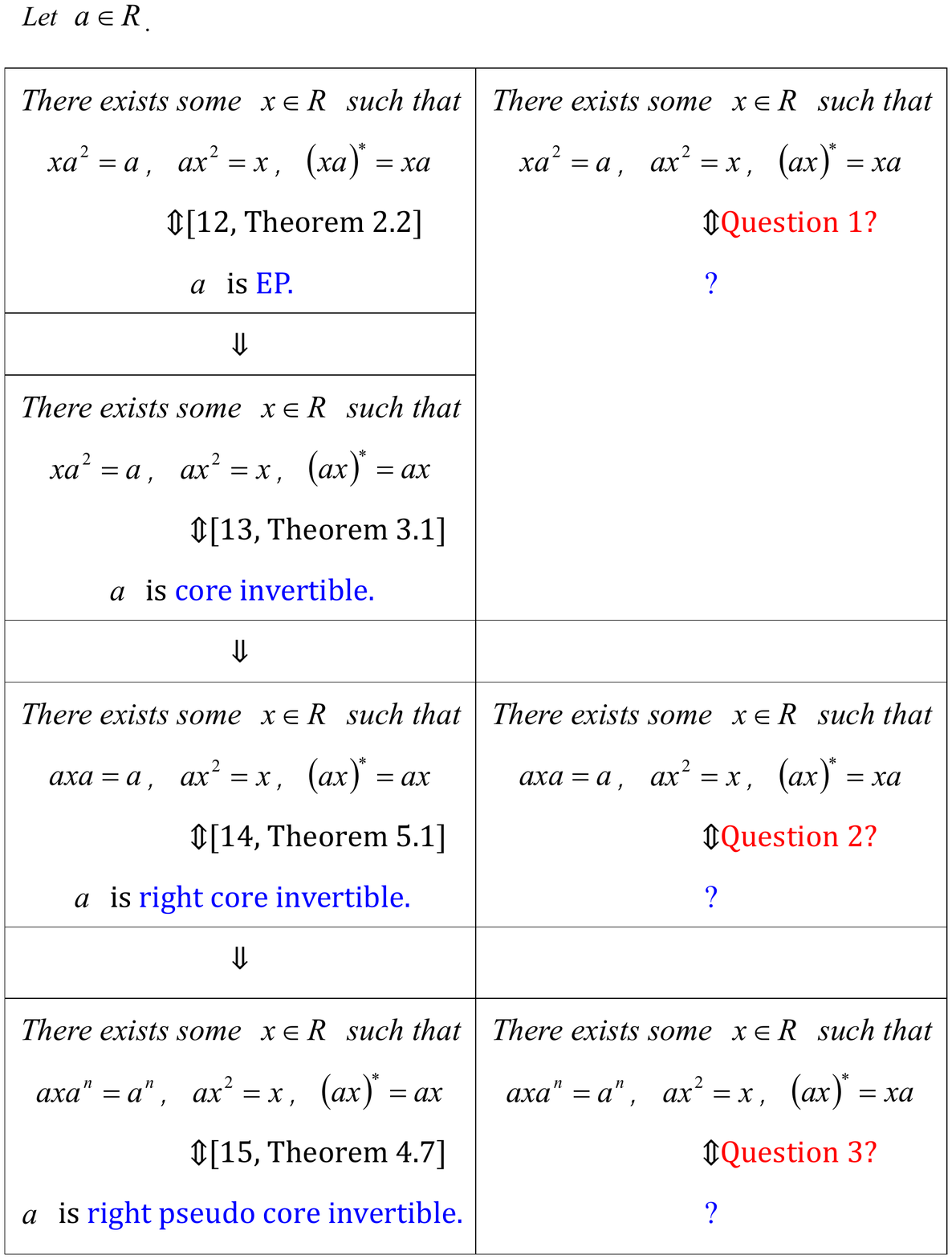}

\end{document}